\documentclass[a4paper]{amsart}
\usepackage[utf8]{inputenc}
\usepackage{amsthm, amssymb}
\usepackage{graphics,tikz}
\usepackage{comment}
\usepackage{hyperref}

\usetikzlibrary{decorations.pathmorphing}
\usetikzlibrary{decorations.pathreplacing}
\usetikzlibrary{arrows}
\usetikzlibrary{arrows, arrows.meta}
\usepackage{caption}
\tikzset{snake it/.style={decorate, decoration=snake}}

\usepackage{dsfont}
\newtheorem{definition}{Definition}[section]
\newtheorem{lemma}[definition]{Lemma}
\newtheorem{theorem}[definition]{Theorem}
\newtheorem{claim}{Claim}
\newtheorem{prop}[definition]{Proposition}

\newtheorem{corollary}[definition]{Corollary}
\newtheorem{question}[definition]{Question}
\newtheorem{problem}[definition]{Problem}
\newtheorem{remark}{Remark}

\newcommand{\mb}{\partial_*}
\newcommand{\N}{\mathbb{N}}

\newcommand{\Z}{\mathbb{Z}}
\newcommand{\mc}{\mathcal}

\usepackage{bbold}

\usepackage{amssymb}
\usepackage{mathtools} 
\usepackage{float}

\theoremstyle{plain}

\begin{document}
\title{(Non-)existence of Cannon--Thurston maps for Morse boundaries}

\author[R. Charney]{Ruth Charney}
    \address{(Ruth Charney) Department of Mathematics, Brandeis University, Waltham, MA, USA}
    \email{charney@brandeis.edu}

\author[M. Cordes]{Matthew Cordes}
    \address{(Matthew Cordes) Department of Mathematics, Heriot-Watt University and Maxwell Institute for Mathematical Sciences, Edinburgh, UK}
    \email{m.cordes@hw.ac.uk}

\author[A. Goldsborough]{Antoine Goldsborough}
    \address{(Antoine Goldsborough) Department of Mathematics, Heriot-Watt University and Maxwell Institute for Mathematical Sciences, Edinburgh, UK}
    \email{antoinegold10@gmail.com}

\author[A. Sisto]{Alessandro Sisto}
    \address{(Alessandro Sisto) Department of Mathematics, Heriot-Watt University and Maxwell Institute for Mathematical Sciences, Edinburgh, UK}
    \email{a.sisto@hw.ac.uk}

\author[S. Zbinden]{Stefanie Zbinden}
    \address{(Stefanie Zbinden) Department of Mathematics, Heriot-Watt University and Maxwell Institute for Mathematical Sciences, Edinburgh, UK}
    \email{sz2020@hw.ac.uk}

\begin{abstract}
    We show that the Morse boundary exhibits interesting examples of both the existence and non-existence of Cannon--Thurston maps for normal subgroups, in contrast with the hyperbolic case.
\end{abstract}

\maketitle

\section{Introduction}

Cannon--Thurston maps, named after the influential work of Cannon and Thurston in \cite{cannonthurston}, are maps induced by the inclusion of a (traditionally infinite index normal) subgroup $H$ of a group $G$ at the level of boundaries $\partial H \to \partial G$. The most well-studied case is where $H$ and $G$ are hyperbolic, and the boundaries here are the Gromov boundaries. In particular Mitra (Mj) \cite{mitra}, proved that Cannon--Thurston maps exist for any normal hyperbolic subgroup of a hyperbolic group. Cannon--Thurston maps have also been considered in the case of relatively hyperbolic groups and their Bowditch boundary \cite{Pal, Bowditch, Mj}, as well as CAT(0) groups and their visual boundaries \cite{CAT0IFP}.

The aim of this paper is to not restrict the class of groups under consideration, in which case one should then consider their Morse boundaries. Morse boundaries were introduced in \cite{charneysultan, cordes} as quasi-isometry invariants of finitely generated groups, and the Morse boundary of a hyperbolic group coincides with its Gromov boundary. We believe it is interesting to study Cannon--Thurston maps in this context, as this is a way of investigating what Morse directions of a subgroup can look like in an ambient group---in particular, in this paper all the subgroups are distorted. Cannon--Thurston maps in this context have been only touched upon in \cite[Proposition 5.5]{CAT0IFP} and \cite[Section 6]{CRSZ}, both of which provide examples of non-existence of Cannon--Thurston maps when the subgroup under consideration is hyperbolic.

In this paper we show that the Morse boundary exhibits interesting examples of both the existence and non-existence of Cannon--Thurston maps for normal subgroups, in contrast with the hyperbolic case.

Regarding non-existence, a case in which a subgroup $H<G$ trivially does not admit a Cannon-Thurston map is the case where $\mb H$ is non-empty, while $\mb G$ is empty; this is the case for $G=H\times \mathbb Z$ for instance. We give a non-trivial example where a Cannon--Thurston map does not exist, where the non-triviality is made precise by the first condition:

\begin{theorem}
\label{thm:main1}
    There are infinitely many automorphisms $\varphi$ of $H=F_2\ast \mathbb Z^2$ such that for $G = H \rtimes_{\varphi} \mathbb Z$ the following hold.
    \begin{enumerate}
        \item (Type-preserving) Every Morse element of $H$ is also a Morse element of $G$, but
        \item (No Cannon-Thurston map) there does not exist a continuous $H$--equivariant map $f :\mb  H \to \mb G$.
    \end{enumerate}
\end{theorem}

In fact, we prove a general criterion for non-existence of such a map in Theorem \ref{thm:no-ct-map}, which we believe can be used in other contexts as well, and we use it to prove Corollary \ref{cor:no_CT_example}, which is a more precise and more general version of the theorem above.

In the opposite direction, if we have $H< G$ with $\mb H$ empty, then trivially there exists a map at the level of Morse boundaries. However, we are not aware of examples in the literature where $H$ is distorted and a map at the level of Morse boundaries exists, except in the hyperbolic case. We provide one such example, and $H$ will again be a right-angled Artin group as in the previous theorem:

\begin{theorem}
\label{thm:main2}
    Let $H$ be the right-angled Artin group with defining graph a path of length $3$. There exists an automorphism $\varphi$ of $H$ such that for $G = H \rtimes_{\varphi} \mathbb Z$ the following hold. There exists a continuous $H$--equivariant map $f :\mb  H \to \mb G$. Moreover, if $\{h_n\in H\}\subseteq H$ and we have $h_n\to h_\infty\in \mb H$, then $h_n\to f(h_\infty)\in \mb G$.
\end{theorem}

As we will see in the proof of Theorem \ref{thm:main2}, the reason that the Cannon-Thurston map exists in this case is the interesting phenomenon that, while $H$ is distorted in $G$, all Morse directions of $H$ are in fact undistorted in $G$.

\subsection{Some open questions}

There are a lot of questions that naturally arise. It seems out of reach to classify when Cannon--Thurston maps exist, and indeed this is not done even in the hyperbolic case. However, it might be possible to do so in concrete cases, for example:

\begin{question}
    Let $H$ be a right-angled Artin group. For which automorphisms $\varphi$ does there exist a continuous $H$--equivariant map $f :\mb  H \to \mb (H\rtimes_{\varphi} \mathbb Z)$?
\end{question}

\begin{problem}
    Finding interesting examples of existence and non-existence of Cannon--Thurston maps in the context of right-angled Coxeter groups.
\end{problem}

In a different direction, there are various phenomena related to the existence of Cannon--Thurston maps that we do not know whether they can happen or not. For instance:

\begin{question}
    Do there exist infinite groups $H\triangleleft G$ where $\mb H$ is empty, but $\mb G$ is non-empty?
\end{question}

For $H$ and $G$ as in the question, a Cannon--Thurston map trivially exists. Notice that $G$ in this case cannot be acylindrically hyperbolic, since infinite normal subgroups of acylindrically hyperbolic groups are acylindrically hyperbolic \cite[Lemma 8.6]{DahmaniGuirardelOsin}, \cite[Corollary 1.5]{Osin:acylindrically}.

Another question of this nature, inspired by \cite[Corollary 6.9]{CRSZ} and Theorem \ref{thm:main1}-(1), is:

\begin{question}
    Does there exist a hyperbolic subgroup $H\triangleleft G$ of a non-hyperbolic group $G$ where every non-trivial element of $H$ is Morse in $G$?
\end{question}

Note that JSJ theory puts strong restrictions on $H$ and $G$ as above, but on the other hand, if one such subgroup exists, a purely pseudo-Anosov but not convex-cocompact subgroup $H$ of the mapping class group of a closed surface $\Sigma$ could give rise to an example. Indeed, given $H$ as above there is a short exact sequence $1\to\pi_1(\Sigma)\to G\to H\to 1$, and $G$ is not hyperbolic but it does not contain $\mathbb Z^2$-subgroups, see \cite{FarbMosher,Hamenstaedt}. This second fact is not sufficient to show that non-trivial elements of $H$ are Morse in $G$, and in fact we do not know how to prove this, but it can be regarded as evidence towards this.

In this introduction we used the terminology ``Cannon--Thurston map'' loosely, but there are some subtleties. For $H<G$ one could consider the existence or non-existence of:
\begin{itemize}
    \item a continuous $H$-equivariant map $f:\mb H\to \mb G$ (as in Theorem \ref{thm:main1}),
    \item a map as above with the additional property that if $\{h_n\in H\}\subseteq H$ and we have $h_n\to h_\infty\in \mb H$, then $h_n\to f(h_\infty)\in \mb G$ (as in Theorem \ref{thm:main2} and \cite{CRSZ}),
    \item a continuous map $f:(H\cup\mb H)\to (G\cup\mb G)$ extending the inclusion $H\to G$, for suitable topologies on $H\cup \mb H$ and $G\cup \mb G$.
\end{itemize}

 To our knowledge, a topology on the union of a group and its Morse boundary has not been defined, even though it probably can be using machinery developed in \cite{cordeshume}.

\begin{question}
    Clarify the relations among the various notions described above. 
\end{question}

\subsection{Outline}

After some preliminaries in Section \ref{sec:prelim}, we prove Theorem \ref{thm:main1} in Section \ref{sec:no_CT}, and Theorem \ref{thm:main2} in Section \ref{sec:CT}.

\subsection*{Acknowledgements}
The authors would like to thank SNSF Ambizione Fellowship of the second author used to organize a research retreat where this paper had its genesis. The second author was supported in part by an SNSF Ambizione Fellowship. The third author was supported by both the EPSRC DTA studentship as well as the LMS Early Career Fellowship. All authors would like to thank the referee for prompt and helpful comments.

\section{Preliminaries}
\label{sec:prelim}
In this section, we will introduce the relevant notions and definitions.

Let $X$ be a geodesic metric space and let $M$ be a Morse gauge, that is, a function $M:\mathbb R_{\geq 1}\times \mathbb R_{\geq 0}\to \mathbb R_{\geq 0}$ which is increasing and continuous in the second factor. Recall that a quasi-geodesic $\gamma$ is $M$--Morse in $X$ if all $(K,C)$--quasi-geodesics $\lambda$ with endpoints on $\gamma$ stay in the $M(K,C)$--neighbourhood of $\gamma$.

In the following, given a finitely generated group $G$, we assume that we have a chosen finite generating set $S$ and identify $G$ with its Cayley graph $\mathrm{Cay}(G, S)$.

 We do not recall the full definition of the Morse boundary $\mb G$ of a group $G$ here, but instead we mention its main features:
\begin{itemize}
    \item $\mb G$ contains so-called strata $\mb^M G$, where $M$ is a Morse gauge, which are compact.
    \item $\mb G$ is the direct limit of the $\mb^M G$.
\end{itemize}
For a detailed definition of the Morse boundary, see \cite{cordes, cordeshume} and the corrigendum \cite{cordessistozbinden}.

\section{Non-existence of Cannon--Thurston maps}
\label{sec:no_CT}
The main goal of this section is to prove a criterion for non-existence of Cannon-Thurston maps. We state this below after a preliminary definition.

\begin{definition}[Weakly Morse]\label{def:weakly-morse}
    We say that a set $A\subset G$ is \emph{weakly $M$--Morse} if $[1, a]$ is $M$--Morse for all $a\in A$. We say that a subset $A\subset G$ is \emph{weakly Morse} if it is weakly $M$--Morse for some Morse gauge $M$. 
\end{definition}

In the statement below we use the notation $H \rtimes_{\varphi} Q$ for the semidirect product given by a homomorphism $\varphi:Q\to\textrm{Aut}(H)$. For $q\in Q$, we denote $\varphi_q=\varphi(q)$.

\begin{theorem}
\label{thm:no-ct-map}
    Let $G = H \rtimes_{\varphi} Q$. Let $q\in Q$ be an infinite-order undistorted element of $Q$, and let $x\in H$ be an infinite order element such that $\varphi_q(x) = x$. Also assume that there exists an element $h\in H$ such that $\{\varphi_q^n(h)\}_{n\in \N}$ is weakly Morse in $H$ and $d_G(1, \varphi_q^n(h))\to \infty$ as $n \to \infty$.  Then there does not exist a continuous $H$--equivariant map $f :\mb  H \to \mb G$.  
\end{theorem}

We now give a preliminary definition, and then summarise how the results in this section yield Theorem \ref{thm:no-ct-map}.

\begin{definition}[Pairwise not Morse]\label{def:pairwise-not-morse}
    We say that a set $A\subset G$ is \emph{pairwise not Morse} if for every Morse gauge $M$ there exists a constant $L\geq 0$ such that for all $a, b\in A$ with $d(a, b)\geq L$ we have that $[a, b]$ is not $M$--Morse. We say that an element $g\in G$ is \emph{pairwise not Morse} if $\langle g \rangle$ is. 
\end{definition}

\begin{proof}[Proof of Theorem \ref{thm:no-ct-map}]
    Proposition \ref{prop:stablized-emelement-implies-pairwise-not-morse} below guarantees that $q$ is pairwise not Morse. Therefore we can apply Proposition \ref{prop:not_weakly_Morse} below to get that $\{\varphi_q^n(h)\}_{n\in \N}$ is not weakly Morse in $G$. We can then use Proposition \ref{prop:no_CT} below with $h_n=\varphi_q^n(h)$ to conclude that there is no continuous $H$--equivariant map $f :\mb  H \to \mb G$.  
\end{proof}

\subsection{Elements fixed by automorphisms}

\begin{prop}\label{prop:stablized-emelement-implies-pairwise-not-morse}
    Let $G = H \rtimes_{\varphi} Q$. Let $q\in Q$ be an infinite-order undistorted element of $Q$. If there exists an infinite order element $h\in H$ such that $\varphi_q(h)=h$, then $q$ is pairwise not Morse.
\end{prop}

\begin{proof}
    Note that if $q$ is an infinite-order undistorted element of $Q$, then if it is not Morse, it is also pairwise not Morse. As shown in e.g. \cite[Lemma 3.8]{CalvezWiest}, if $q$ is Morse, then $\langle q \rangle$ has finite index in its centralizer. If $\varphi_q(h)=h$, then $\langle h \rangle$ is in the centralizer of $q$ and hence $q$ cannot be Morse.
\end{proof}

\subsection{Iterates are not weakly Morse}

\begin{prop}
\label{prop:not_weakly_Morse}
    Let $G = H \rtimes_{\varphi} Q$. Assume that there exists $q \in Q$ which is pairwise not Morse and also assume that there exists an element $h\in H$ such that $\{\varphi_q^n(h)\}_{n\in \N}$ is weakly Morse in $H$ and $d_G(1, \varphi_q^n(h))\to \infty$ as $n \to \infty$. Then the set $\{\varphi_q^n(h)\}_{n\in \N}$ is not weakly Morse in $G$. 
\end{prop}

\begin{proof}
    Note that $d_G(1, \varphi_q^n(h))\to \infty$ as $n \to \infty$ implies that $q$ has infinite order.
    
    Assume by contradiction that $\{\varphi_q^n(h)\}_{n\in \N}$ is weakly $M$--Morse in $G$ for some Morse gauge $M$. In fact, we can assume that any subsegment of $[1, \varphi_q^n(h)]$ is $M$-Morse (see for example Lemma 2.1 of \cite{cordes}). Let $C = M(3, 0)$ and let $M'$ be the Morse gauge such that in every quadrangle where one side has length at most $C$, two of the other sides are $M$-Morse, the fourth side is $M'$--Morse (such an $M'$ exists by \cite[Lemma 2.3]{cordes}). Let $L\geq 4 d_G(1, q)$ be the length such that for every pair of points $q^k, q^l$ with $d_G(q^k, q^l)\geq L$ we have that $[q^k, q^l]$ is not $M'$--Morse. Such an $L$ exists since $q$ is not pairwise Morse. Lastly, let $n\in \N$ be such that $d_G(1, \varphi_q^n(h)) \geq 10^9L$ and let $k$ be such that $L\leq d_G(1, q^k)\leq 2L$. 

    Let $x$ and $y$ be the points on $[q^k, q^k \varphi^{n-k}(h)]$ closest to $1$ and $\varphi^n(h)$ respectively. By Lemma 2.9 of \cite{zbinden:graphOfGroups} we have that $\lambda = [1, x]\circ[q^k, q^k \varphi^{n-k}(h)]_{[x, y]} \circ [y, \varphi^n(h)]$ is a $(3, 0)$--quasi-geodesic. Since $[1, \varphi^n(h)]$ is $M$--Morse by assumption, there exists a point $z$ on $[1, \varphi^n(h)]$ with $d_G(z, x)\leq M(3, 0) = C$. Now consider the quadrangle $\mc Q = ([1, \varphi^n(h)]_{[1, z]}, [z, x], [ q^k \varphi^{n-k}(h), q^k]_{[x, q^k]}, [q^k, 1])$. Note that by assumption, $[1, \varphi^n(h)]_{[1, z]}$ and $[ q^k \varphi^{n-k}(h), q^k]_{[x, q^k]}$ are $M$--Morse and $d(x, z)\leq C$. Thus, by the choice of $M'$, we know that the fourth side $[1, q^k$] is $M'$--Morse. This is a contradiction to the choice of $L$, concluding the proof.
\end{proof}

\subsection{A criterion for non-existence of Cannon-Thurston maps}
\begin{prop}
\label{prop:no_CT}
     Let $G = H \rtimes_{\varphi} Q$. If there exists a sequence of distinct elements $(h_n)_{n\in \N}$ of $H$ such that $\{h_n\}_{n\in \N}$ is weakly Morse in $H$ but not weakly Morse in $G$, then there does not exist a continuous $H$--equivariant map $f :\mb  H \to \mb G$.
\end{prop}

\begin{proof}
    Assume by contradiction that there exists a continuous $H$--equivariant map $f: \mb H \to \mb G$. Let $a\in \mb H$. 
    Since $\{h_n\}_n$ is weakly Morse in $H$, there exists a Morse gauge $M$ such that $h_na\in \mb^M H$ for all $n\in \N$ (this is because if two sides of a triangle are Morse, then the third one is as well, quantitatively, see \cite[Lemma 2.3]{cordes}). We can assume that $M$ is refined in the sense of \cite{cordessistozbinden}, so that the stratum $\mb ^M H$ is compact. The sequence $(h_n\cdot a)_n$ has a converging subsequence. In fact, up to passing to a subsequence, we can assume that $(h_n\cdot a)_n$ converges to a point $a_{\infty}\in \mb H$. 

    Since $f$ is continuous and $H$--equivariant, the sequence $(h_n\cdot f(a))_n$ converges to $f(a_{\infty})$. In particular, there exists a Morse gauge $M'$ such that $h_n\cdot f(a)\in \mb^{M'}G$ for all $n\in \N$ (indeed the union of the sequence and its limit is compact so the sequence is contained in a stratum by \cite[Lemma 4.1]{cordesdurham}). 

    Again by \cite[Lemma 2.3]{cordes}, there exists a Morse gauge $M''$ such that $[1, h_n]$ is $M''$--Morse in $G$ for all $n$. This is a contradiction to $\{h_n\}_{n\in \N}$ not being weakly Morse in $G$.
\end{proof}

\subsection{An example}
We now give an example where we can apply Theorem \ref{thm:no-ct-map}. 
\begin{corollary}\label{cor:no_CT_example}
    Let $H=H_0\ast \mathbb Z^2$, where $H_0$ is a hyperbolic group, and consider an automorphism $\varphi : H \to H$ with $\varphi|_{H_0}:H_0\to H_0$ a hyperbolic automorphism, and $\varphi|_{\mathbb Z^2}$ the identity. Let $G = H \rtimes_{\varphi} \mathbb Z$. Then
    \begin{enumerate}
        \item (Type-preserving) Every Morse element of $H$ is also a Morse element of $G$, but
        \item (No Cannon-Thurston map) there does not exist a continuous $H$--equivariant map $f :\mb  H \to \mb G$.
    \end{enumerate}
\end{corollary}

\begin{proof}
    Non-existence of the Cannon-Thurston map follows directly from Theorem \ref{thm:no-ct-map}, so we have to show that Morse elements of $H$ are still Morse in $G$.

    First of all, $H$ is hyperbolic relative to $\mathbb Z^2$, and therefore we claim:
    
    \begin{claim}
        An element $h$ of $H$ is Morse if and only if it has infinite order and is not conjugate into $\mathbb Z^2$.
    \end{claim}

    \begin{proof}
        Morse elements in any group have infinite order and cannot be contained in $\mathbb Z^2$ subgroups, see e.g. \cite[Lemma 3.8]{CalvezWiest}, so we only need to show the ``if'' part. By \cite[Corollary 1.7]{Osin:elementary}, we have that $h$ is contained in an elementary subgroup $E(h)$ with the property that $G$ is hyperbolic relative to $\{\Z^2,E(g)\}$. Therefore, $E(g)$ is a peripheral subgroup of a relatively hyperbolic structure on $H$, and peripheral subgroups are strongly quasiconvex by \cite[Lemma 4.15]{DrutuSapir}. This readily implies Morseness of $h$ (since $\langle h\rangle$ and $E(h)$ lie at finite Hausdorff distance).
    \end{proof}
    
    Since $\psi=\varphi|_{H_0}$ is hyperbolic, we have that $G_0=H \rtimes_{\psi} \mathbb Z$ is hyperbolic by \cite{BestvinaFeighn}. Using the decomposition $G=G_0\ast_{\langle t\rangle} \mathbb Z^3$ and the combination theorem for relative hyperbolicity \cite{Dahmani:combination}, we then see that $G_0$ is hyperbolic relative to the mapping torus $\mathbb Z^3$ of $\mathbb Z^2<H$. If we have a Morse element $h$ of $H$, in $G$ it is not conjugate into $\mathbb Z^3$, and therefore it is Morse in $G$ (again by \cite{Osin:elementary,DrutuSapir}), concluding the proof.
\end{proof}

\begin{remark} 
    We can compare this with the case when $H$ is a free group and $\varphi$ is the atoroidal automorphism defined by $a \mapsto aba$ and $b \mapsto ba$. In this case the Morse Cannon--Thurston map does not exist \cite[Proposition 5.5]{CAT0IFP}. A standard calculation shows that $[a,b]$ commutes with $\langle t \rangle$ in $G = F_2 \rtimes \langle t \rangle$. So here, the map is not type preserving, and that is where the failure of the Cannon-Thurston map is exhibited. In fact, all examples covered by \cite{CAT0IFP} fail to be type preserving.
\end{remark}

\section{A mapping torus of a RAAG with a Cannon-Thurston map}
\label{sec:CT}
Let $H = \langle a,b,c,d \mid [a,b], [b,c] [c,d] \rangle$,  and $\varphi \in \mathrm{Aut}(H)$ defined by $a \mapsto ab,\ 
    b \mapsto b, \
    c \mapsto c, \
    d \mapsto d.
$
Let $G = H \rtimes_{\varphi} \mathbb Z$ and observe that $H$ is distorted in $G$ since, for example, $a^nb^n=ta^nt^{-1}$.

\begin{theorem} \label{thm:CT}
    There exists a continuous $H$--equivariant map $f :\mb  H \to \mb G$. Moreover, if $\{h_n\in H\}\subseteq H$ and we have $h_n\to h_\infty\in \mb H$, then $h_n\to f(h_\infty)\in \mb G$.
\end{theorem}

Both $G$ and $H$ have useful graph of group decompositions:
$H = \langle a,b,c \rangle \ast_{\langle b, c \rangle}\langle b,c,d \rangle$ and $G = \langle a,b,c, t \rangle \ast_{\langle b, c, t \rangle}\langle b,c,d, t \rangle$. Let $T$ be the Bass--Serre tree associated to this decomposition of $H$ and let $T'$ be the Bass--Serre tree associated to this decomposition of $G$.

\begin{lemma} \label{isom embedding}
    There exists a $G$-equivariant isometry $T \to T'$.
\end{lemma}

\begin{proof}
    Recall that vertices (and edges) of a Bass--Serre tree are cosets of vertex (and edge) groups. Let $\iota \colon T \to T'$ be defined by the natural inclusions: $\langle a,b,c \rangle \hookrightarrow \langle a,b,c, t \rangle$, $\langle b,c \rangle \hookrightarrow \langle b,c,t \rangle$, and $\langle b,c,d \rangle \hookrightarrow \langle b,c,d, t \rangle$.

    It follows from the definition of this map that this is an equivariant isometric embedding.

    We now need to show the map is surjective. First, note that given an element $g \in G$, one can always use the the relations $txt^{-1}=x$ and $t^{-1}xt=x$ for $x \in \{b,c,d\}$ and $t a t^{-1}= ab$ to bubble all of the occurrences of $t^{\pm1}$'s in $g$ to the end of $g$ so that $g= h \cdot t^{\pm n}$ in $G$  for some $n\in  \mathbb{Z}_{\geq 0}$ and some $h \in H$. So given a coset $g\langle a,b,c,t\rangle$ (or $g\langle b,c,d, t\rangle$), it follows that the coset $h \langle a,b,c \rangle$ (or $h \langle b,c,d \rangle$) is mapped to $g\langle a,b,c,t\rangle$ (or $g\langle b,c,d, t\rangle$).
\end{proof}

We will use the previous lemma to prove that $G$ acts acylindrically on $T'$

\begin{lemma}
\label{lem:acyl}
    The action of $G$ on $T'$ is acylindrical. 
\end{lemma}

\begin{proof}
    We wish to show that the stabilizer of any long path in $T'$ only stabilizes a bounded subtree. Thus we wish to understand the intersection of any two conjugates of $\langle b, c,t \rangle$ in $\langle a,b,c,t \rangle$ (or similarly  any two conjugates of $\langle b, c,t \rangle$ in $\langle b,c,d,t \rangle$). We claim that $\langle b \rangle$ is the only subgroup in the intersection of any two conjugates. Assuming the claim, the acylindricity of the action holds because $\langle b \rangle$ only stabilizes a bounded subtree of $T$ and by Lemma \ref{isom embedding}, also only stabilizes a bounded subtree of $T'$.

    First note that $a c a^{-1} \not \subset \langle c \rangle$, so to prove the claim, it is enough to show that the intersection of any two conjugates of $\langle b, c,t \rangle$ in $\langle a,b,c,t \rangle$ do not contain $\langle t \rangle$. 

    To see this, we consider another graph of groups decomposition of the subgroup $\langle a,b,c,t \rangle < G$: 
        \[\langle a,b,c,t \rangle =\langle a, b, t \rangle \ast_{\langle b,t \rangle} \langle b,c,t \rangle.\]
    We note that $\langle a, b, t\rangle$ is isomorphic to the Heisenberg group where $b$ generates the center. Since any element of $\langle a, b, t\rangle$ can be written as $a^n t^m b^ k$ for some integers $n,m,k$, to show that the intersection of any two conjugates of $\langle b, c,t \rangle$ do not contain $\langle t \rangle$ we just need check that $a^{-n} t a^n \not \subset \langle t \rangle$. A quick calculation shows that $a^{-n} t a^n = b^nt$ and finishes the proof.
\end{proof}

Fix word metrics on $H$ and $G$. Given a Morse gauge $M$, let $H^{(M)}$ be the set of all $h\in H$ such that a geodesic $[1,h]$ is $M$-Morse, and we will use a similar notation for $G$. Recall that a map $f:X\to Y$ between metric spaces is $M$-\emph{stable}, for $M$ a Morse gauge, if for all $x_1,x_2\in X$ some geodesic $[f(x_1),f(x_2)]$ is $M$-Morse.

The following is the key result to prove the existence of a Cannon--Thurston map. It tells us that the Morse directions in $H$ are undistorted and Morse in $G$. 

\begin{prop}
\label{prop:qi_stable}
    For every Morse gauge $M$ there exist a Morse gauge $M'$ and a constant $K$ such that the inclusion $\iota_M:H^{(M)}\to G$ is a $M'$-stable $(K,K)$--quasi-isometric embedding.
\end{prop}

\begin{proof}
    We have that $H^{(M)}$ quasi-isometrically embeds into $T$ by \cite[Lemma 4.5]{cordescharneysisto}. Therefore, in view of Lemma \ref{isom embedding}, we have that the composition of $\iota_M$ and an orbit map $\pi$ of $G$ to $T'$ is a quasi-isometric embedding. Since $\pi$ is coarsely Lipschitz, $\iota_M$ needs to be a quasi-isometric embedding. Stability follows from acylindricity of the action, Lemma \ref{lem:acyl}, which ensures that geodesics $\gamma$ in $G$ with the property that $\pi\circ\gamma$ is a quasi-geodesic are Morse (with Morse gauge controlled in terms of the hyperbolicity constant); this is\cite[Lemma 4.6]{cordescharneysisto} (see also \cite{DrutuMozesSapir,sisto}). 
\end{proof}

In order to show that the inclusion $H \hookrightarrow H \rtimes \mathbb{Z}$ induces a continuous map on the Morse boundaries, we shall use the following lemma, which is a consequence of the universal property of direct limits.

\begin{lemma}
\label{lem:criterion_continuity}
    Let $\{X_i\}$ and $\{Y_j\}$ be two sequences of nested topological spaces. Let 
        \[X=\mathop{\lim_{\longrightarrow}}_{i}X_i \quad \text{and} \quad  Y=\mathop{\lim_{\longrightarrow}}_{j}Y_j\] 
    be the direct limits of  $\{X_i\}$ and $\{Y_j\}$. A map $f:X \to Y$ is continuous if the following hold:
    \begin{itemize}
        \item for each $i$ there exists a $j$ such that $f(X_i) \subseteq Y_j$,
        \item $f\vert_{X_i}\to Y_j$ is a continuous map.
    \end{itemize}
\end{lemma}

\begin{proof}[Proof of Theorem \ref{thm:CT}]
    By Proposition \ref{prop:qi_stable}, and the fact that stable quasi-isometric embeddings induce maps at the level of Morse boundaries (see \cite{cordeshume}), for any Morse gauge $M$ there exists a Morse gauge $M'$ and a continuous map $f^M:\mb^M H=\partial H^{(M)}\to \mb^{M'} G$ with the property that if $\{h_n\}\subseteq H^{(M)}$ and $h_n\to p\in \mb^MH$, then $h_n\to f^M(p)\in \mb^{M'}G$; this property uniquely determines $f^M$.

    By Lemma \ref{lem:criterion_continuity} there exists a continuous map $f:\mb H\to \mb G$ which restricts to $f^M$ on $\mb^MH$, for all Morse gauges $M$. This map satisfies the ``moreover'' part by construction, and it is readily seen that this necessarily makes it equivariant.
\end{proof}

\bibliography{main}
\bibliographystyle{alpha}
\end{document}